\pdfoutput=1
\documentclass{article}
\usepackage[utf8]{inputenc}
\usepackage{amsmath}
\usepackage{amssymb}
\usepackage{amsthm}
\usepackage{tikz-cd}
\usepackage{mathtools}
\usepackage{hyperref}
\usepackage{comment}
\usepackage[titletoc,title]{appendix}
\hypersetup{colorlinks=true}
\usepackage[totalwidth=480pt, totalheight=680pt]{geometry}

\title{The Rouquier Dimension of Quasi-Affine Schemes}
\author{Noah Olander}
\date{}

\newtheorem{thm}{Theorem}

\newtheorem{prop}{Proposition}
\newtheorem{lemma}{Lemma}

\theoremstyle{definition}

\newtheorem*{remark}{Remark}

\begin{document}

\maketitle 

\begin{abstract}
    We prove that for $X$ a regular quasi-affine scheme of dimension $d$, $\mathcal{O}_X$ is a $d$-step generator of $D^b_{coh}(X)$, establishing Orlov's conjecture in this case. We prove something weaker in the projective case. The main techniques are a spectral sequence argument borrowed from topology and the converse ghost lemma, both suitably adapted to work in this setting. Along the way we prove that on a regular scheme $X$ of dimension $d < \infty$ any composition of $d+1$ morphisms of $D^b_{coh}(X)$ which are zero on cohomology sheaves is zero.
\end{abstract}

\section{Introduction}

In this paper we prove that for a regular quasi-affine scheme $X$ of dimension $d < \infty$, $D^b_{coh}(X) = \langle \mathcal{O}_X \rangle _{d+1}$, which implies Orlov's conjecture \cite[Conjecture 10]{Orl09} in the quasi-affine case. This is already known when $X$ is affine by \cite[Proposition 3.3]{ELAGIN2021334}, but the quasi-affine case does not follow since a regular quasi-affine scheme need not be an open of a regular affine scheme. In particular, if $X$ is a regular projective variety over a field and $U$ is the affine cone over $X$ with the vertex removed, our result implies that $\mathrm{Rdim}(U) = \mathrm{dim}(U) = \mathrm{dim}(X)+1$. As this may suggest, our methods also show something in the (quasi-)projective case. Namely, for a regular quasi-projective scheme of dimension $d$ we will see that every object of $D^b_{coh}(X)$ can be built from the objects $\{\mathcal{O}_X(n)\}_{n \in \mathbf{Z}}$ using at most $d$ cones. If one could show that only finitely many of the $\mathcal{O}_X(n)$ sufficed (as is the case on a smooth projective curve by \cite{Orl09}), this would prove Orlov's conjecture. Of course, that is exactly the difficult part, but we still hope that this paper will be useful in clarifying what is hard about Orlov's conjecture and what is easy.

Along the way we prove Theorem \ref{thm-ghost} which is interesting in its own right. It says that if $X$ is a Noetherian regular scheme of dimension $d$ and $K_0 \to K_1 \to \cdots \to K_{d+1}$ are maps in $D^b_{coh}(X)$ which are zero on cohomology sheaves, then the composition $K_0 \to K_{d+1}$ is zero. The argument proving this occurs in the proof of \cite[Proposition 4.5]{CHRISTENSEN1998284}, so it is possible that Theorem \ref{thm-ghost} is known to experts but we have not found a reference. 

Next we show how Theorem \ref{thm-ghost} leads to a simple proof of Orlov's conjecture for quasi-affine schemes. Using the fact that $\mathcal{O}_X$ is ample on a quasi-affine scheme $X$, we show that if a morphism $K \to L$ in $D^b_{coh}(X)$ is an $\mathcal{O}_X$-ghost, then it is zero on cohomology sheaves. Hence in this case, Theorem \ref{thm-ghost} implies that the composition of $d+1$ $\mathcal{O}_X$-ghosts vanishes, so we should be able to conclude by some version of the converse ghost lemma \cite[Theorem 4]{article}. In fact we do not quite meet the hypotheses of loc. cit. so we supply a simple argument which proves a different form of the converse ghost lemma. 

The author would like to thank Johan de Jong for many enlightening conversations about Rouquier dimension including a very useful one concerning the converse ghost lemma. The author would also like to point out the PhD thesis \cite{letz2020generation} which contains similar arguments to the ones in this paper in the affine setting.

\section{Proofs of the Theorems}
Let $\mathcal{A}$ be an abelian category and $\varphi: K \to L$ a map in $D( \mathcal{A})$. We would like to know if $\varphi = 0$. An obvious necessary condition is that $H^n(\varphi) : H^n(K) \to H^n(L)$ be zero for all $n$, but this is not sufficient: Consider any nonzero map $A \to A[1]$ with $A \in \mathcal{A}$. However it turns out that if all the maps $H^n(\varphi)$ vanish, we may associate to $\varphi$ a sequence of Ext classes $\xi_n \in \mathrm{Ext}^1_{\mathcal{A}} (H^n(K) , H^{n-1}(L))$ whose vanishing gives another necessary condition for $\varphi$ to be zero. One may expect that if all the $\xi_n$ vanish then we could associate to $\varphi$ a sequence of $\mathrm{Ext}^2$-classes, and so forth. This is not true and the failure is caused by non-zero differentials in a spectral sequence. The following proposition gives a corrected statement. We let $FAb$ denote the category of filtered abelian groups and $Ab$ the category of abelian groups. 

\begin{prop}
\label{prop-obstructions}
Let $\mathcal{A}$ be an abelian category with enough injectives. Then for each $K, L \in D^b(\mathcal{A})$ there is a decreasing filtration $F$ on $\mathrm{Hom}_{D(\mathcal{A})} (K, L)$, which is natural in the sense that it arises from a functor $D^b(\mathcal{A})^{opp} \times D^b(\mathcal{A}) \to FAb$ whose composition with the forgetful functor $FAb \to Ab$ is the $\mathrm{Hom}$-functor, and which satisfies for $K, L, M \in D^b(\mathcal{A})$:
    \renewcommand\labelenumi{(\theenumi)}
    \begin{enumerate}
        \item $F^0\mathrm{Hom}_{D(\mathcal{A})}(K, L) = \mathrm{Hom}_{D(\mathcal{A})}(K, L)$ and $F^p\mathrm{Hom}_{D(\mathcal{A})}(K, L) = 0$ for $p \gg 0$.
        \item If $f \in F^p\mathrm{Hom}_{D(\mathcal{A})}(K, L)$ and $g \in F^q \mathrm{Hom}_{D(\mathcal{A})} (L, M)$ then $g \circ f \in F^{p+q}\mathrm{Hom}_{D(\mathcal{A})}(K, M)$.
        \item $F^p\mathrm{Hom}_{D(\mathcal{A})}(K, L)/F^{p+1}\mathrm{Hom}_{D(\mathcal{A})}(K, L)$ is a subquotient of $\prod _{n \in \mathbf{Z}}\mathrm{Ext}^p_{\mathcal{A}} (H^n(K), H^{n-p}(L))$.
        \item $F^1\mathrm{Hom}_{D(\mathcal{A})}(K, L) = \{\varphi \in \mathrm{Hom}_{D(\mathcal{A})} (K, L) : H^n(\varphi) = 0 \text{ for all } n \in \mathbf{Z}\}$. 
       
    \end{enumerate}
\end{prop}

\begin{remark}
The proposition gives an enrichment of $D^b(\mathcal{A})$ in the symmetric monoidal category of filtered abelian groups; recall that the tensor product in this category is defined by $F^k(A \otimes B) = \sum _{i + j = k} \mathrm{Im} (F^iA \otimes F^jB \to A \otimes B)$.
\end{remark}

We prove Proposition \ref{prop-obstructions} in Appendix \ref{section-appendix}. The proof uses a spectral sequence
$$
E_1^{p,q} = \prod_{n \in \mathbf{Z}} \mathrm{Ext}^{2p+q} (H^n(K), H^{n-p}(L)) \implies \mathrm{Ext}^{p+q}(K, L),
$$
and is not difficult but we need to use the definition of the spectral sequence to prove (2).

\begin{thm}
\label{thm-ghost}
Let $X$ be a Noetherian regular scheme of dimension $d < \infty$. Let $K_0 \to K_1 \to \cdots \to K_{d+1}$ be morphisms in $D^b_{coh} (X)$ whose induced maps on cohomology sheaves vanish. Then the composition $K_0 \to K_{d+1}$ is zero. 
\end{thm}

\begin{proof}
Note that $D^b_{coh} (X)$ is a full subcategory of $D(\mathrm{QCoh}(X))$ by \cite[\href{https://stacks.math.columbia.edu/tag/09T4}{Tag 09T4}]{stacks-project} and $\mathrm{QCoh}(X)$ is an abelian category with enough injectives. Thus we may use the filtration 
$$
\mathrm{Hom}(K_0, K_{d+1}) = F^0 \supset F^1 \supset F^2 \supset \cdots
$$
of Proposition \ref{prop-obstructions}. By \cite[\href{https://stacks.math.columbia.edu/tag/0FZ3}{Tag 0FZ3}]{stacks-project} we have $\mathrm{Ext}^{i}(\mathcal{F}, \mathcal{G}) = 0$ for $i > d$ and $\mathcal{F}, \mathcal{G}$ coherent sheaves on $X$. Therefore by (3) of Proposition \ref{prop-obstructions}, $F^{d+1} = F^{d+2} = \cdots.$ Since $F^p = 0$ for $p \gg 0$, in fact $F^{d+1} = 0$. Then by (4) each $K_i \to K_{i +1}$ is in $F^1\mathrm{Hom} (K_i , K_{i+1}),$ so that by (2) the composition $K_0 \to K_{d+1}$ is in $F^{d+1}\mathrm{Hom}(K_0, K_{d+1}) = 0$ and we are done.
\end{proof}

\begin{lemma}
\label{lemma-colim}
Let $X$ be a quasi-compact and quasi-separated scheme. Let $\mathcal{L}$ be an invertible sheaf and $s \in H^0(X, \mathcal{L}).$ 
Then for $P \in D_{perf} (X)$ and $K \in D_{QCoh}(X)$,
$$
\mathrm{Hom}_{D(X_s)}(P|_{X_s}, K|_{X_s}) = \mathrm{colim}_{n\geq 0} \mathrm{Hom}_{D(X)}(P \otimes^{\mathbf{L}}_{\mathcal{O}_X}\mathcal{L}^{\otimes -n}, K),
$$
where $X_s = \{s \neq 0\}$ and the transition maps are given by multiplication by $s$.
\end{lemma}

\begin{proof}
Denote $j: X_s \to X$ the immersion. Since $j$ is affine and using the formula $R_f = \mathrm{colim} ( R \xrightarrow{f} R \xrightarrow{f} \cdots )$,
$$
Rj_*(\mathcal{O}_{X_s}) = j_*(\mathcal{O}_{X_s}) = \mathrm{colim}_{n \geq 0} \mathcal{L}^{\otimes n} = \mathrm{hocolim}_{n \geq 0}\mathcal{L}^{\otimes n},
$$ where the transition maps are given by multiplication by $s$. Then by the projection formula \cite[\href{https://stacks.math.columbia.edu/tag/08EU}{Tag 08EU}]{stacks-project} and since derived tensor products commute with homotopy colimits (because they commute with direct sums) we get $Rj_*(K|_{X_s}) = K \otimes ^{\mathbf{L}}_{\mathcal{O}_X} Rj_* (\mathcal{O}_{X_s}) = \mathrm{hocolim}_{n \geq 0} K \otimes ^{\mathbf{L}}_{\mathcal{O}_X} \mathcal{L}^{\otimes n}$. Since $P$ is compact in $D_{QCoh}(X)$ (\cite[Theorem 3.1.1]{BvdB}),
$$
\mathrm{Hom}_{D(X_s)}(P|_{X_s}, K|_{X_s}) = \mathrm{Hom}_{D(X)}(P, Rj_*(K|_{X_s})) = \mathrm{colim}_{n \geq 0}\mathrm{Hom}_{D(X)}(P, K \otimes _{\mathcal{O}_X}^{\mathbf{L}} \mathcal{L}^{\otimes n}).
$$
Moving $\mathcal{L}^{\otimes n}$ to the other side yields the result.
\end{proof}

\begin{lemma}
\label{lemma-ample-sheaf}
Let $X$ be a scheme with an ample invertible sheaf $\mathcal{L}$. Let $K \in D_{QCoh} (X)$. Then there is a set $I$ and a morphism $\bigoplus_{i \in I} \mathcal{L}^{\otimes m_i}[n_i] \to K$ with $n_i, m_i \in \mathbf{Z}$, $m_i < 0$, which is surjective on cohomology sheaves. If $K$ is bounded and each cohomology sheaf is of finite type (for example if $X$ is Noetherian and $K \in D^b_{coh}(X)$), then $I$ may be taken finite. 
\end{lemma}


\begin{proof}
It suffices to show there exists a set $I$ and a morphism $\bigoplus_{i \in I}\mathcal{L}^{\otimes m_i} \to K$ which is surjective on $H^0$. 
By assumption, there are sections $s_{i} \in H^0(X, \mathcal{L}^{\otimes k_i}) , i = 0 , \dots , r$ with $k_i > 0$ such that the open subschemes $X_{s_i} = \{s_i \neq 0\}$ are affine and cover $X$. Since $H^0(K)$ is quasi-coherent and $X_{s_i}$ is affine, we may choose sections $t_{ij} \in H^0(X_{s_i}, H^0(K))$ which generate $H^0(K)$ over $X_{s_i}$. Again using that $X_{s_i}$ is affine, we have
$$
\mathrm{Hom}(\mathcal{O}_{X_{s_i}}, K|_{X_{s_i}}) = H^0(X_{s_i}, H^0(K)),
$$
so for each $i,j$ there is a morphism $\varphi_{ij}: \mathcal{O}_{X_{s_i}} \to K|_{X_{s_i}}$ which on $H^0$ is given by the section $t_{ij}$. Then by Lemma \ref{lemma-colim} there exists an $N_{ij}<0$ and a morphism $\psi_{ij}: \mathcal{L}^{\otimes k_i \cdot N_{ij}} \to K$ whose restriction to $X_{s_i}$ is $\varphi_{ij}$. 
Now the morphism $\bigoplus_{i,j} \mathcal{L}^{\otimes k_i \cdot N_{ij} } \to K$ given by $(\psi_{ij})$ does the job. Note that if $H^0(K)$ is a finite type quasi-coherent sheaf then we need only finitely many sections $t_{ij}$, which proves the second statement.
\end{proof}

\begin{remark}
If instead of assuming $X$ has an ample invertible sheaf we assume only that $X$ has an ample family $\{\mathcal{L}_k \}$ of invertible sheaves, the same argument shows that for every $K \in D_{QCoh}(X)$ there is a set $I$ and a morphism $\bigoplus _i \mathcal{L}_{k_i} ^{\otimes m_i}[n_i] \to K$ with $m_i < 0$ which is surjective on cohomology sheaves, and that $I$ may be taken finite if $K$ is bounded with finite type cohomology sheaves.
\end{remark}

\begin{thm}
\label{thm-main}
Let $X$ be a Noetherian regular scheme of dimension $d < \infty$. Assume $X$ has an ample invertible sheaf $\mathcal{L}$. Then $D^b_{coh}(X) = \langle \{\mathcal{L}^{\otimes n}\}_{n \leq 0}\rangle _{d+1}$. In particular, if $X$ is quasi-affine, then $D^b_{coh}(X) = \langle \mathcal{O}_X \rangle _{d+1}$. 
\end{thm}

Thus when $X$ is quasi-affine, $\mathrm{Rdim}(D^b_{coh}(X)) \leq d$. If $X$ is also of finite type over a field, we have $\mathrm{Rdim}(D^b_{coh}(X)) = d$ by \cite[Proposition 7.17]{rouquier_2008}. 

\begin{proof}
Note that the second statement follows from the first since $\mathcal{O}_X$ is ample on a quasi-affine scheme. 
Let's prove the first. Let $K \in D^b_{coh} (X)$. We will show $K \in \langle \{\mathcal{L}^{\otimes n}\}_{n \leq 0}\rangle _{d+1}$. Set $K = K_0.$ Choose a finite set $I$ and a morphism $\bigoplus_{i \in I} \mathcal{L}^{\otimes m_i} [n_i] \to K$ as in Lemma \ref{lemma-ample-sheaf} and let $K_1$ be the cone. Note that $K_0 \to K_1$ is zero on cohomology sheaves by construction, and its cone is in $\langle \{\mathcal{L}^{\otimes n}\}_{n \leq 0}\rangle _{1}$. Now repeat the process with $K = K_1$ and so on to obtain a sequence 
$$
K_0 \to K_1 \to \cdots \to K_{d+1}
$$
such that each $K_i \to K_{i+1}$ is zero on cohomology sheaves and has cone in $\langle \{\mathcal{L}^{\otimes n}\}_{n \leq 0}\rangle _{1}$. Thus $K_0 \to K_{d+1}$ is zero by Theorem \ref{thm-ghost}. We will prove by induction that the cone of $K_0 \to K_i$ is in $\langle \{\mathcal{L}^{\otimes n}\}_{n \leq 0}\rangle _{i}$. For $i = 1$ this is known and for $i = d+1$ this proves the Theorem: Since $K = K_0 \to K_{d+1}$ is zero the cone is isomorphic to $K_{d+1} \oplus K [1]$ and the category $\langle \{\mathcal{L}^{\otimes n}\}_{n \leq 0}\rangle _{d+1}$ is closed under direct summands and shifts. 

So assume known that the cone of $K_0 \to K_i$ is in $\langle \{\mathcal{L}^{\otimes n}\}_{n \leq 0}\rangle _{i}$. Then by the octahedral axiom there is a distinguished triangle
$$
C \to D \to E \to C[1]
$$
with $C$ a cone of $K_0 \to K_i$, $D$ a cone of $K_0 \to K_{i+1}$, and $E$ a cone of $K_i \to K_{i+1}$. Since $C \in \langle \{\mathcal{L}^{\otimes n}\}_{n \leq 0}\rangle _{i}$ and $E \in \langle \{\mathcal{L}^{\otimes n}\}_{n \leq 0}\rangle _{1}$ it follows that $D \in \langle \{\mathcal{L}^{\otimes n}\}_{n \leq 0}\rangle _{i+1}$, as needed. 
\end{proof}

\begin{remark}
If $X$ is a Noetherian regular scheme with affine diagonal, $X = \bigcup_{i} U_i$ is a finite affine open covering, and $D_i = X \setminus U_i$ with the reduced induced subscheme structure, then $\{\mathcal{O}_X(D_i)\}_i$ is an ample family of invertible sheaves on $X$. The argument above, with Lemma \ref{lemma-ample-sheaf} replaced by the remark following it, shows that $D^b_{coh}(X) = \langle \{\mathcal{O}_X(n D_i)\}_{i, n\leq 0} \rangle _{d+1}$. 
\end{remark}

\begin{appendices}

\section{Enrichment of \texorpdfstring{$D^b(\mathcal{A})$}{Db(A)}}
\label{section-appendix}

Let $\mathcal{A}$ be an abelian category with enough injectives. As in \cite[I, Chapter V]{cotangent}, let $DF(\mathcal{A})$ denote the filtered derived category of $\mathcal{A}$, whose objects are represented by $\mathcal{A}$-complexes $K^\bullet$ with a \emph{finite} decreasing filtration (i.e., there exists $i > 0$ such that for all $n$, $F^iK^n = 0$ and $F^{-i}K^n = K^n$). The full subcategory of $DF(\mathcal{A})$ spanned by those $K$ with $\mathrm{gr}^i(K) \in D^+(\mathcal{A})$ for each $i \in \mathbf{Z}$ is denoted $D^+F(\mathcal{A})$, and $D^bF(\mathcal{A}), D^-F(\mathcal{A})$ are defined similarly. We will let $Ab$ denote the category of abelian groups and $FAb$ the category of filtered abelian groups. Given $K \in DF(\mathcal{A})$ we will also write $K$ for its image in $D(\mathcal{A})$ but when taking Hom groups we will use subscripts to make it clear which category we are viewing the objects in.


\begin{lemma}
\label{lemma-dfil-is-enhanced}
For each $K \in DF(\mathcal{A})$ and $L \in D^+F(\mathcal{A}),$ there is a decreasing filtration $F$ on $\mathrm{Hom}_{D(\mathcal{A})} (K, L)$, which is natural in the sense that it arises from a functor (*) fitting into a commutative diagram
$$
\begin{tikzcd}
 DF(\mathcal{A}) ^{opp} \times D^+F(\mathcal{A}) \ar[r, "(*)"] \ar[d, "forget"]  &FAb \ar[d, "forget"] \\
 D(\mathcal{A}) ^{opp} \times D^+(\mathcal{A}) \ar[r]   &Ab, \\
 (K, L) \arrow[r, mapsto ] & \mathrm{Hom}_{D(\mathcal{A})}(K, L)
\end{tikzcd}
$$
and which satisfies for $K \in DF(\mathcal{A}), L, M \in D^+F(\mathcal{A})$:
    \renewcommand\labelenumi{(\theenumi)}
    \begin{enumerate}
        \item The filtration on $\mathrm{Hom}_{D(\mathcal{A})}( K, L)$ is finite. 
        \item If $f \in F^p\mathrm{Hom}_{D(\mathcal{A})}(K, L)$ and $g \in F^q \mathrm{Hom}_{D(\mathcal{A})} (L, M)$ then $g \circ f \in F^{p+q}\mathrm{Hom}_{D(\mathcal{A})}(K, M)$.
        \item $F^p\mathrm{Hom}_{D(\mathcal{A})}(K, L)/F^{p+1}\mathrm{Hom}_{D(\mathcal{A})}(K, L)$ is a subquotient of $\prod _{n \in \mathbf{Z}}\mathrm{Hom}_{D(\mathcal{A})} (\mathrm{gr}^n(K), \mathrm{gr}^{n+p}(L))$.
        \item Assume $\mathrm{Ext}^n_{D(\mathcal{A})}(\mathrm{gr}^i(K), \mathrm{gr}^j(L)) = 0$ when $n = 0,-1$ and $i>j$, for example if 
        
        \noindent $\mathrm{Ext}^n_{D(\mathcal{A})}(\mathrm{gr}^i(K)[-i], \mathrm{gr}^j(L)[-j]) = 0$ for $n< 0$. Then:
        \begin{enumerate}
            \item $F^0\mathrm{Hom}_{D(\mathcal{A})}(K, L) = \mathrm{Hom}_{D(\mathcal{A})}(K,L) = \mathrm{Hom}_{DF(\mathcal{A})}(K, L)$.
            \item $F^1\mathrm{Hom}_{D(\mathcal{A})}(K, L) = \{\varphi \in \mathrm{Hom}_{D(\mathcal{A})}(K, L) = \mathrm{Hom}_{DF(\mathcal{A})} (K, L) : \mathrm{gr}^n(\varphi) = 0 \text{ for all } n \in \mathbf{Z}\}$.
        \end{enumerate}
    \end{enumerate}
\end{lemma}
\begin{proof}
The commutative diagram in the statement is the outer square of a commutative diagram:
$$
\begin{tikzcd}
 DF(\mathcal{A}) ^{opp} \times D^+F(\mathcal{A}) \ar[r, "(**)"] \ar[d, "forget"] & DF(\mathbf{Z}) \ar[d, "forget"]  \ar[r, "(***)"] &FAb \ar[d, "forget"] \\
 D(\mathcal{A}) ^{opp} \times D^+(\mathcal{A}) \ar[r]  &D(\mathbf{Z}) \ar[r, "H^0"] &Ab. \\
 (K, L) \arrow[r, mapsto ] & R\mathrm{Hom}(K, L)
\end{tikzcd}
$$
We need to explain the two top horizontal arrows, but note first that the composition of the bottom two arrows is indeed the $\mathrm{Hom}$-functor.

The functor $(**)$ is the filtered variant of $R\mathrm{Hom}$ constructed in \cite{cotangent}. Let us briefly recall its construction: For $K \in DF(\mathcal{A})$ and $L \in D^+F(\mathcal{A})$ we may represent $K$ by a complex $K^\bullet$ with a finite filtration and we may represent $L$ by a complex $L^\bullet$ with a finite filtration such that each $F^iL^n$ is an injective object of $\mathcal{A}$ and $L^n = 0$ for $n \ll 0$ (we say $L^\bullet$ is \emph{of injective type}). Then the Hom complex $\mathrm{Hom}^\bullet (K^\bullet , L^\bullet)$ whose $i^{th}$ term is
$$
\mathrm{Hom}^i(K^\bullet, L^\bullet) = \prod_{n \in \mathbf{Z}} \mathrm{Hom}(K^n, L^{n+i})
$$
has a filtration such that $(f^n) \in \prod_{n \in \mathbf{Z}} \mathrm{Hom}(K^n, L^{n+i})$ lies in $F^p$ iff
$$
f^n(F^jK^n) \subset F^{j+p}L^{n+i} \text{ for all } n, j.
$$
Since the filtrations on $K^\bullet$ and $L^\bullet$ were assumed finite, the filtration on $\mathrm{Hom}^\bullet (K^\bullet, L^\bullet)$ is also finite so that we have indeed defined an object of $DF(\mathbf{Z})$ which we denote simply as $R\mathrm{Hom}(K,L)$. Since $L^\bullet$ is a bounded below complex of injectives, the underlying object of $D(\mathbf{Z})$ is the usual $R\mathrm{Hom}(K,L)$.

The functor (***) takes an object $K$ to $H^0(K)$ equipped with its induced filtration, namely
$$
F^i H^0(K) = \mathrm{Im}(H^0(F^iK) \to H^0(K)).
$$
Note that by our convention that filtrations are finite, $(***)$ lands in the full subcategory of $FAb$ consisting of abelian groups with a finite filtration, so that (1) is immediate. 

Let us prove (2). Represent $K, L, M$ by finitely filtered complexes $K^\bullet, L^\bullet, M^\bullet$ with $L^\bullet$ and $M^\bullet$ of injective type. Then 
$$
f \in \mathrm{Im}(H^0(F^p\mathrm{Hom}^\bullet (K^\bullet, L^\bullet)) \to H^0(\mathrm{Hom}^\bullet (K^\bullet, L^\bullet))
$$
so we can represent $f$ by a morphism of complexes $\varphi: K^\bullet \to L^\bullet$ such that $\varphi(F^jK^\bullet) \subset F^{j+p}L^\bullet$ for each $j$, and similarly we can represent $g$ by a morphism of complexes $\psi: L^\bullet \to M^\bullet$ such that $\psi(F^jL^\bullet) \subset F^{j+q}L^\bullet$ for each $j$. Then the composition $\psi \circ \varphi: K^\bullet \to M^\bullet$ represents $g \circ f : K \to M$ and satisfies $\psi \circ \varphi (F^j K^\bullet) \subset F^{j+p+q}M^\bullet$ for each $j$, thus $g \circ f \in F^{p+q}\mathrm{Hom}(K,M)$, as needed.

For (3) we consider the spectral sequence associated to the filtered complex $\mathrm{Hom}^\bullet(K^\bullet, L^\bullet)$ constructed in the second paragraph. It converges since the filtration is finite, and it has
$$
E_1^{p,q} = H^{p+q}(\mathrm{gr}^p(R\mathrm{Hom}(K,L))) = \prod_{n \in \mathbf{Z}} \mathrm{Ext}^{p+q}_{D(\mathcal{A})}(\mathrm{gr}^n(K), \mathrm{gr}^{n+p}(L))
$$
(see \cite[V 1.4.9]{cotangent}) and
$$
E_{\infty}^{p,-p} = F^p\mathrm{Hom}_{D(\mathcal{A})}(K, L)/F^{p+1}\mathrm{Hom}_{D(\mathcal{A})}(K, L)
$$
directly from the definition of the spectral sequence of a filtered complex. Since $E_{\infty}^{p,-p}$ is a subquotient of $E_1^{p,-p}$ this proves (3).

The first equality of (4)(a) follows from (3) since under the assumptions the group in (3) vanishes for $p<0$. The second equality is \cite[Proposition 3.1.4(i)]{BBD}. For (4)(b) consider the distinguished triangle 
$$
F^1R\mathrm{Hom}(K,L) \to F^0R\mathrm{Hom}(K,L) \to \mathrm{gr}^0R\mathrm{Hom}(K,L) \to .
$$
The map $H^0(F^0R\mathrm{Hom}(K,L)) \to H^0(\mathrm{gr}^0R\mathrm{Hom}(K,L))$ identifies with the canonical map $\mathrm{Hom}_{DF(\mathcal{A})}(K, L) \to \prod_{n \in \mathbf{Z}} \mathrm{Hom}_{D(\mathcal{A})}(\mathrm{gr}^n(K) , \mathrm{gr}^n(L)),$ see \cite[V 1.4.6 and 1.4.9]{cotangent} hence using the exact sequence of cohomology,
\begin{align*}
F^1\mathrm{Hom}_{D(\mathcal{A})}(K, L) &= \mathrm{Im}(H^0(F^1R\mathrm{Hom}(K,L)) \to H^0(R\mathrm{Hom}(K,L))) \\
&= \mathrm{Im}(H^0(F^1R\mathrm{Hom}(K,L)) \to H^0(F^0R\mathrm{Hom}(K,L))) \\
&= \mathrm{Ker}(\mathrm{Hom}_{DF(\mathcal{A})}(K, L) \to \prod_{n \in \mathbf{Z}} \mathrm{Hom}_{D(\mathcal{A})}(\mathrm{gr}^n(K) , \mathrm{gr}^n(L)).
\end{align*}
The second equality is (4)(a).
\end{proof}

\begin{proof}[Proof of Proposition \ref{prop-obstructions}]
There is a fully faithful functor $can: D^b(\mathcal{A}) \to D^bF(\mathcal{A})$ whose composition with the forgetful functor $D^bF(\mathcal{A}) \to D^b(\mathcal{A})$ is isomorphic to the identity and such that $\mathrm{gr}^p ( can(K) ) = H^{-p}(K)[p]$. It is the quasi-inverse of the equivalence of categories of \cite[Proposition 3.1.6]{BBD}. Informally, $can$ equips $K$ with its canonical filtration $F^iK = \tau_{\leq -i}K$. For $K,L \in D^b(\mathcal{A})$ we view them in $D^bF(\mathcal{A})$ via $can$ and therefore Lemma \ref{lemma-dfil-is-enhanced} gives the desired filtration on $\mathrm{Hom}_{D(\mathcal{A})}(K, L)$. All of (1)-(4) are now immediate.
\end{proof}

\end{appendices}

\newpage

\bibliographystyle{alpha}
\bibliography{references}

\textsc{Columbia University Department of Mathematics, 2990 Broadway, New York, NY 10027}

\href{mailto:nolander@math.columbia.edu}{nolander@math.columbia.edu}
\end{document}